\documentclass{CSML}
\usepackage{verbatim}
\usepackage{bbm}
\usepackage{hyperref}
\usepackage{color}

\newcommand{\dOi}{13(1:1)2017}

\lmcsheading{}{1--11}{}{}{Jan.~21, 2016}{Jan.~23, 2017}{}

\hypersetup{
    colorlinks=true,
    linkcolor=black,
    citecolor=black,
    filecolor=black,
    urlcolor=black,
}

\title[Logical compactness and constraint satisfaction problems]
{Logical compactness and constraint satisfaction problems}
\author[D. Rorabaugh]{Danny Rorabaugh\rsuper a}
\address{{\lsuper{a}}Department of Mathematics and Statistics,
Queen's University,
Kingston, Ontario, Canada}
\author[C. Tardif]{Claude Tardif\rsuper b}
\address{{\lsuper{b,c}}Department of Mathematics and Computer Science,
Royal Military College of Canada,
Kingston, Ontario, Canada}
\author[D. Wehlau]{David Wehlau\rsuper c}
\address{\vspace{-18pt}}
\newcommand{\idm}[1]{\mbox{\rm id}_{#1}}
\newfont{\Bb}{msbm10 scaled\magstep1}

\newcommand{\bA}{\mathbb{A}}
\newcommand{\bB}{\mathbb{B}}
\newcommand{\bC}{\mathbb{C}}

\newcommand{\bK}{\mathbb{K}}

\newcommand{\choice}[1]{$\mbox{\rm Choice}(#1)$}
\newcommand{\kw}[1]{$\mbox{\rm KW}(#1)$}

\begin{document}

\begin{abstract} We investigate a correspondence between
the complexity hierarchy of constraint satisfaction problems
and a hierarchy of logical compactness hypotheses for finite relational
structures. It seems that the harder a constraint satisfaction
problem is, the stronger the corresponding compactness hypothesis is.
At the top level, the NP-complete constraint satisfaction problems
correspond to compactness hypotheses that are equivalent
to the ultrafilter axiom in all the cases we have investigated. 
At the bottom level, the simplest constraint satisfaction problems 
correspond to compactness hypotheses that
are readily provable from the axioms of Zermelo and Fraenkel.
\end{abstract}

\keywords{Compactness theorem, Relational structures, Constraint satisfaction problems}

\maketitle

\section{Introduction} 

A relational structure $\bA$ is said to be {\em compact}
if for any structure $\bB$ of the same type, 
$\bB$ admits a homomorphism to $\bA$ whenever 
every finite substructure $\bB'$ of $\bB$ admits 
a homomorphism to $\bA$. The compactness theorem
of logic implies that every finite structure is compact. 
However, the compactness theorem is equivalent to the ultrafilter 
axiom, a consequence of the axiom of choice 
that is not provable from the axioms of Zermelo and Fraenkel. 
In this paper, we restrict our attention to the case where 
$\bA$ is finite, but we do not assume the ultrafilter axiom. 
Instead, for each structure $\bA$, we consider the  
statement ``$\bA$ is compact'' as an hypothesis 
that is consistent with the axioms of Zermelo and Fraenkel, 
but not necessarily provable from these.

It turns out that the strength of such compactness hypotheses
varies widely.  For some structures, compactness can be proved from 
the axioms of Zermelo and Fraenkel. At the other extreme,
for some other structures, compactness implies the ultrafilter axiom,
hence the compactness of all finite structures. Perhaps it makes 
sense to call the latter structures ``compactness-complete''.
This designation is borrowed from that of complexity classes,
but our results also parallel complexity results:
The structures $\bA$ which we prove to be 
compactness-complete have their corresponding constraint 
satisfaction problems NP-complete. In contrast, the 
structures $\bA$ for which we show
that the statement  ``$\bA$ is compact'' is provable 
from the axioms of Zermelo and Fraenkel are the structures 
of ``width one''. The corresponding constraint 
satisfaction problems are arguably the simplest
polynomial cases. Also, the compactness of a structure 
implies that of any structure 
which can be primitively positively defined from it. 
Therefore it is possible that the algebraic approach 
to the complexity classification of constraint satisfaction 
problems would be relevant to the study of the hierarchy 
of compactness hypotheses as well.
 
The ``dichotomy conjecture'' of 
Feder and Vardi~\cite{federvardi} states that every constraint satisfaction problem is 
either polynomial or NP-complete. There is no such dichotomy
between the compactness hypotheses that are equivalent
to the ultrafilter axiom and those that are provable from
the axioms of Zermelo and Fraenkel. Indeed we will exhibit
structures with polynomial constraint satisfaction problems,
for which compactness hypotheses are equivalent to some
intermediate ``cardinal-specific'' versions of the axiom of choice.

The axiom of choice and its variants grew out of the
need to distinguish between existential and constructive
aspects of mathematical proofs. In time, tools such as 
forcing theory became available to establish independence 
results. The possibility of connecting such independence
results with constraint satisfaction problems
is the main motivation of our investigation.

The paper is structured as follows. The next section will present
the basics on homomorphisms of relational structures and 
an alternative characterisation of compactness. In the following two
sections, we will present our results on compactness hypotheses derivable
from the axioms of Zermelo and Fraenkel, and on compactness hypotheses 
equivalent to the ultrafilter axiom. Up to then, the relevant set-theoretic 
concepts are pretty standard, but afterwards we will present 
lesser known axioms weaker than the ultrafilter axiom, before moving on to
intermediate compactness hypotheses. 

\section{Relational structures and compactness} \label{tolerant}  
A {\em type} is a finite set $\sigma = \{R_1,\dots,R_m\}$ of 
{\em relation symbols}, each with  an {\em arity} 
$r_i$ assigned to it. 
A $\sigma$-structure is a relational structure 
$\bA = (A;R_1(\bA),\dots,R_m(\bA))$ where $A$ 
is a nonempty set called the {\em universe} of $\bA$, and 
$R_i(\bA)$ is an $r_i$-ary relation on $A$ for each $i$.
Let $\bA$ and $\bB$ be $\sigma$-structures, with universes 
$A$ and $B$ respectively.
A {\em homomorphism} of $\bB$ to $\bA$ is a map 
$f: B \rightarrow A$ such that $f(R_i(\bB)) 
\subseteq R_i(\bA)$ for all $i=1,\dots,m$.
If $B$ is a subset of $A$ and the identity is
a homomorphism from $\bB$ to $\bA$,
then $\bB$ is called a {\em substructure} of $\bA$.
The constraint satisfaction problem associated to a structure $\bA$
is the problem of determining whether an input structure admits 
a homomorphism to $\bA$.

As stated in the introduction, a relational structure $\bA$ is
called {\em compact} if for any structure $\bB$ of the same type, 
$\bB$ admits a homomorphism to $\bA$ whenever 
every finite substructure $\bB'$ of $\bB$ admits 
a homomorphism to $\bA$. The main result of this section
is an alternative characterisation of compact structures
in terms of ``filter-tolerant'' powers. Recall
that a filter $\mathcal{F}$ on a set $I$ is a family of nonempty
subsets of $I$ that is closed under intersection and
contains every superset of each of its members. 
A filter which is maximal
with respect to inclusion is called an ultrafilter. 
Equivalently, a filter is an
ultrafilter if and only if it contains precisely 
one of each pair of complementary
subset of $I$. The ultrafilter axiom states that 
every filter extends to an ultrafilter.

Let $\mathcal{F}$ be a filter on a set $I$, and $\mathbb{A}$ a 
relational structure of some type $\sigma$. 
The {\em $\mathcal{F}$-tolerant power} 
$\mathbb{A}^I_{\mathcal{F}}$ is the structure defined as follows.
The universe of $\mathbb{A}^I_{\mathcal{F}}$ is the set
$A^I$ of all functions of $I$ to the universe $A$ of $\mathbb{A}$,
and for each $R \in \sigma$ of arity $k$, 
$R(\mathbb{A}^I_{\mathcal{F}}) \subseteq (A^I)^k$
is the set of all $k$-tuples $(f_1, \ldots, f_k)$ such that the 
set $\{i \in I | (f_1(i), \ldots, f_k(i)) \in R(\mathbb{A}) \}$
belongs to $\mathcal{F}$.

\begin{prop} \label{com=ft}
Let $\mathbb{A}$ be a finite relational structure. Then $\mathbb{A}$
is compact if and only if for every set $I$ and every filter
$\mathcal{F}$ on $I$, $\mathbb{A}^I_{\mathcal{F}}$ admits
a homomorphism to $\mathbb{A}$.
\end{prop}
\begin{proof}
First, suppose that $\mathbb{A}$ is compact.
Let $\mathbb{B}$ be a finite substructure of
$\mathbb{A}^I_{\mathcal{F}}$. For every
$R$ in $\sigma$ and every $(f_1, \ldots, f_k) \in R(\mathbb{B})$,
the set
$$S_{R,(f_1, \ldots, f_k)} =
\{ i \in I | (f_1(i), \ldots, f_k(i)) \in R(\mathbb{A}) \}$$
is an element of $\mathcal{F}$. Thus, the finite intersection
$$S_{\mathbb{B}} = \bigcap \{ S_{R,(f_1, \ldots, f_k)} |
R\in \sigma, (f_1, \ldots, f_k) \in R(\mathbb{B}) \}$$
is an element of  $\mathcal{F}$ hence it is not empty.
For every $i \in S_{\mathbb{B}}$, 
the map $\phi: \mathbb{B} \rightarrow \mathbb{A}$
defined by $\phi(f) = f(i)$ is a homomorphism, hence 
$\mathbb{B}$ admits a homomorphism to $\mathbb{A}$.
If $\mathbb{A}$ is compact, this implies that
$\mathbb{A}^I_{\mathcal{F}}$ admits a homomorphism to $\mathbb{A}$.

Now suppose that $\mathbb{A}^I_{\mathcal{F}}$ admits
a homomorphism to $\mathbb{A}$ for every filter $\mathcal{F}$
on a set $I$. Let $\mathbb{B}$ be a structure such that
every finite substructure of $\mathbb{B}$ admits a homomorphism to 
$\mathbb{A}$.
Let $I$ be the set of all maps from the universe of $\mathbb{B}$
to that of $\mathbb{A}$. For $R \in \sigma$ and $(x_1, \ldots, x_k) \in R(\mathbb{B})$,
let 
$$
(R,(x_1, \ldots, x_k))^+ = \{ i \in I | (i(x_1), \ldots, i(x_k)) \in R(\mathbb{A}) \}.
$$
For a finite collection $\{ (R_j,(x_{1,j}, \ldots, x_{k_j,j}))^+ | j = 1, \ldots, \ell\}$
of these sets, let $\mathbb{B}'$ be the substructure of $\mathbb{B}$ spanned by
$B' = \bigcup_{j=1}^{\ell} \{ x_{1,j}, \ldots, x_{k_j,j}\}$.
Then $\mathbb{B}'$ is a finite substructure of $\mathbb{B}$. By
hypothesis, there exists a homomorphism $\phi: \mathbb{B}' \rightarrow \mathbb{A}$. 
Any extension of such a homomorphism $\phi$ to the universe of $\mathbb{B}$
belongs to $\bigcap \{ (R_j,(x_{1,j}, \ldots, x_{k_j,j}))^+ | j = 1, \ldots, \ell\}$.
Thus, the family 
$\{ (R,(x_1, \ldots, x_k))^+ | R \in \sigma, (x_1, \ldots, x_k) \in R(\mathbb{B}) \}$
generates a filter $\mathcal{F}$ on $I$. By construction, the natural
map $\psi: \mathbb{B} \rightarrow \mathbb{A}^I_{\mathcal{F}}$ 
defined by $\psi(x) = f$, where $f(i) = i(x)$, is a homomorphism.
If there exists a homomorphism 
$\phi: \mathbb{A}^I_{\mathcal{F}} \rightarrow \mathbb{A}$,
then $\phi \circ \psi: \mathbb{B} \rightarrow \mathbb{A}$
is a homomorphism.
\end{proof}

In view of this result, the standard derivation of the compactness
of $\mathbb{A}$ from the ultrafilter axiom is a direct consequence
of the following result.

\begin{prop} \label{uf->hom}
If $\mathcal{F}$ is contained in an ultrafilter, 
then for every finite structure $\mathbb{A}$,
$\mathbb{A}^I_{\mathcal{F}}$ admits a homomorphism to $\mathbb{A}$.
\end{prop}
\begin{proof}
Let $\mathcal{U}$ be an ultrafilter containing $\mathcal{F}$.
Then for every $f$ in $A^I$, there is a unique $x$ in $A$
such that $f^{-1}(x) \in \mathcal{U}$. We put $\phi(f) = x$. 
We show that $\phi$ is a homomorphism. Let $R \in \sigma$ be a 
relation of arity $k$, and $(f_1, \ldots, f_k) \in R(\mathbb{A}^I_{\mathcal{F}})$.
Then 
$$S = \{i \in I| (f_1(i), \ldots, f_k(i)) \in R(\mathbb{A}) \} \in \mathcal{F}
\subseteq \mathcal{U}.$$
Therefore the set $S \cap \left ( \bigcap_{j=1}^k f_j^{-1}(\phi(f_j)) \right )$ is
in $\mathcal{U}$, hence it is not empty. For 
$i \in S \cap \left ( \bigcap_{j=1}^k f_j^{-1}(\phi(f_j)) \right )$, we have
\begin{equation}
(\phi(f_1), \ldots, \phi(f_k)) = (f_1(i), \ldots, f_k(i)) \in R(\mathbb{A}). \tag*{\qEd}
\end{equation}
\def\popQED{}
\end{proof}

Our results on compactness will use the characterisation
of Proposition~\ref{com=ft}, or its variant in terms of the natural
quotient discussed next.
Let $\sim$ be the equivalence defined on the universe of
$\mathbb{A}^I_{\mathcal{F}}$ by $f \sim g$ if 
$\{i \in I | f(i) = g(i)\} \in {\mathcal{F}}$.
Note that when $\mathcal{F}$ is an ultrafilter, 
$\mathbb{A}^I_{\mathcal{F}}/\!\sim$
is the standard ultrapower construction. 

In general, for $(f_1, \ldots, f_k ) \in R$ and $g_j \sim f_j , 
j = 1, \ldots, k$, we have $(g_1 , \ldots , g_k ) \in R$.
Thus, $\mathbb{A}^I_{\mathcal{F}}$ is the ``lexicographic sum'' of
the equivalence classes of $\sim$.
Any homomorphism from $\mathbb{A}^I_{\mathcal{F}}/\!\sim$ to $\mathbb{A}$
can be composed with the quotient map of $\mathbb{A}^I_{\mathcal{F}}$ to 
$\mathbb{A}^I_{\mathcal{F}}/\!\sim$. Conversely,
when $\mathbb{A}$ is finite, any homomorphism 
$\phi: \mathbb{A}^I_{\mathcal{F}} \rightarrow \mathbb{A}$ 
allows to define a
homomorphism $\psi: \mathbb{A}^I_{\mathcal{F}}/\!\sim \rightarrow \mathbb{A}$
as follows: given an ordering of the
universe of $\mathbb{A}$, define $\psi(x/\!\sim)$ to be the smallest 
$a$ such that there exists $y \in x/\!\sim$ with $\phi(y) = a$. 
Therefore we have the following
\begin{prop} \label{com=ftq}
Let $\mathbb{A}$ be a finite relational structure. Then $\mathbb{A}$
is compact if and only if for every set $I$ and every filter
$\mathcal{F}$ on $I$, $\mathbb{A}^I_{\mathcal{F}} /\!\sim$ admits
a homomorphism to $\mathbb{A}$.
\end{prop}

\section{Structures of width one}
Feder and Vardi~\cite{federvardi} described various heuristics for
finding homomorphisms to given relational structures. For such
a heuristic, it is desirable to characterise the class of structures
for which it is decisive. The simplest among these heuristics is 
the consistency-check algorithm of width one. This is the
intuitive heuristic that it is unconsciously adopted by Sudoku 
puzzle solvers with no scientfic background. 
The structures for which width one consistency-check is decisive are called structures of 
width one, or structures with tree duality. We refer the reader
to~\cite{federvardi} for details. We will use the non-algorithmic
structural characterisation given below.

For a structure $\mathbb{A}$, let $\mathcal{P}(\mathbb{A})$
be the structure defined as follows. 
The universe of $\mathcal{P}(\mathbb{A})$
is the set of nonempty subsets of the universe $A$ of $\mathbb{A}$.
For $R \in \sigma$ of arity $k$, $R(\mathcal{P}(\mathbb{A}))$
consists of the $k$-tuples $(S_1, \ldots, S_k)$ such that 
$\mbox{pr}_i(R(\mathbb{A}) \cap (\Pi_{i=1}^k S_i)) = S_i$ 
for $i = 1, \ldots, k$. (Where $\mbox{pr}_i$ is the $i$-th projection.)
A structure $\mathbb{A}$ is said to have 
{\em width one} if there exists a homomorphism from $\mathcal{P}(\mathbb{A})$ 
to $\mathbb{A}$. 
\begin{prop} 
Every structure of width one is compact. 
\end{prop}
\begin{proof}
Let $\mathbb{A}$ be a structure of width one.
We will show that for every filter $\mathcal{F}$
on a set $I$, there exists a homomorphism from
$\mathbb{A}^I_{\mathcal{F}}$ to $\mathbb{A}$.

For $f$ in the universe of $\mathbb{A}^I_{\mathcal{F}}$, put 
$$\mathcal{S}(f) = \{ S \subseteq A : f^{-1}(S) \in \mathcal{F}\}$$
(where $A$ is the universe of $\bA$).
In particular, $A \in \mathcal{S}(f)$, $\emptyset \not \in \mathcal{S}(f)$,
and $\mathcal{S}(f)$ is closed under intersections. 
Since $A$ is finite, $\mathcal{S}(f)$ contains a smallest
member $\phi(f)$.
We show that $\phi: \mathbb{A}^I_{\mathcal{F}} \rightarrow \mathcal{P}(\mathbb{A})$
is a homomorphism. 

Let $R \in \sigma$ be a relation of arity $k$,
and $(f_1, \ldots, f_k)$ an element of $R(\mathbb{A}^I_{\mathcal{F}})$.
Put 
$$S = \{ i \in I : (f_1(i), \ldots, f_k(i)) \in R(\mathbb{A})\}.$$
Then $S$ is in $\mathcal{F}$, hence 
$T = S \cap \left (\bigcap_{i=1}^k f_i^{-1}(\phi(f_i)) \right ) \in \mathcal{F}$.
For $i = 1, \ldots, k$, we have $f_i(T) \subseteq \phi(f_i)$,
but by minimality of $\phi(f_i)$, the inclusion cannot be strict since 
$f_i^{-1}(f_i(T)) \in \mathcal{F}$. Thus every
$x \in \phi(f_i)$ is the $i$-th coordinate of some
$(f_1(t), \ldots, f_k(t))$ in $R(\mathbb{A})$, with $t \in T$ 
and $f_j(t) \in \phi(f_j)$ for $j = 1, \ldots, k$. In other words,
$\mbox{pr}_i(R(\mathbb{A}) \cap (\Pi_{i=1}^k \phi(f_i))) = \phi(f_i)$ 
for $i = 1, \ldots, k$, whence 
$(\phi(f_1), \ldots, \phi(f_k)) \in R(\mathcal{P}(\mathbb{A}))$.
This shows that $\phi$ is a homomorphism. 

Since $\mathbb{A}$ has width one, there exists a homomorphism
$\psi$ of $\mathcal{P}(\mathbb{A}) \rightarrow \mathbb{A}$.
The composition
$\psi \circ \phi: \mathbb{A}^I_{\mathcal{F}}\rightarrow \mathbb{A}$
is then a homomorphism.
Thus, for every filter $\mathcal{F}$
on a set $I$, there exists a homomorphism from
$\mathbb{A}^I_{\mathcal{F}}$ to $\mathbb{A}$.
By Proposition~\ref{com=ft}, this implies that $\mathbb{A}$
is compact.
\end{proof}
It is not clear whether other structures can be proved to be compact
within the axioms of Zermelo-Fraenkel.
\begin{prob}
Let $\mathbb{A}$ be a structure for which compactness follows
from the axioms of Zermelo-Fraenkel. Does $\mathbb{A}$
necessarily have width one?
\end{prob}

\section{Compactness results equivalent to the ultrafilter axiom} 
In graph theory, the De Bruijn-Erd\H{o}s theorem is the statement that a
graph is $k$-colourable if and only if all of its finite subgraphs
are $k$-colourable. In our terminology, this is the statement
that the complete graphs are compact.  (The complete graph $\bK_n$ on $n$ vertices is 
the structure with universe $\{0, \ldots, n-1\}$ and the binary adjacency relation $\neq$.)
Various proofs were known in the early fifties. Then in 1971, L\"{a}uchli~\cite{lauchli} 
proved that the ultrafilter axiom is a consequence of the compactness of $\bK_n$
for any $n \geq 3$.

In this section, we present our proof of L\"{a}uchli's result, and extend it
to many relational structures using primitive positive definability. Incidentally,
the complete graphs with at least three elements correspond to the first constraint 
satisfaction problems that were shown to be NP-complete.

\begin{lem} \label{hom->uf}
Let $n \geq 3$ be an integer and $\mathcal{F}$ a filter on set $I$.
If $\left(\bK_n\right)^I_{\mathcal{F}}$
admits a homomorphism to $\bK_n$,
then $\mathcal{F}$ is contained in an ultrafilter.
\end{lem}
\begin{proof}
Let $\phi: \left(\bK_n\right)^I_{\mathcal{F}} \rightarrow \bK_n$ be a
homomorphism.  We write $\idm{k}$ 
for the constant function with constant value $k \in \{0, \ldots, n-1\}$. 
Since the restriction of $\phi$ to the constant functions is bijective,
we can assume without loss of generality that $\phi(\idm{k}) = k$ for
all $k \in \{0, \ldots, n-1\}$.
If $n > 3$,
every element of $\left(\bK_3\right)^I_{\mathcal{F}} \subseteq 
\left(\bK_n\right)^I_{\mathcal{F}}$ is adjacent to $\idm{k}$ for all
$k \geq 3$. Therefore the restriction of $\phi$ to $\left(\bK_3\right)^I_{\mathcal{F}}$
is an idempotent homomorphism to $\bK_3$. Therefore we can assume that $n=3$.

For $X \subseteq I$, we write $\mathbbm{1}_X$ for the characteristic 
map of $X$, that is, $\mathbbm{1}_X(i) = 1$ if $i \in X$ and 
$\mathbbm{1}_X(i) = 0$ otherwise.
Then $\mathbbm{1}_X$ is adjacent to $\idm{2}$, therefore 
$\phi(\mathbbm{1}_X) \in \{0,1\}$ for all $X \subseteq I$. Put  
$$\mathcal{U} = \{ X | \phi(\mathbbm{1}_X) = 1\}.$$
We will show that $\mathcal{U}$ is an ultrafilter containing 
$\mathcal{F}$.

For $F \in \mathcal{F}$, $\mathbbm{1}_F$ is adjacent to $\idm{0}$. 
Since $\phi(\idm{0}) = 0$, we have $\phi(\mathbbm{1}_F) = 1$. Thus, 
$\mathcal{F} \subseteq \mathcal{U}$. 
Also, for any $X \subseteq I$,
$\mathbbm{1}_X$, $\mathbbm{1}_{\overline{X}}$ and $\idm{2}$ are
mutually adjacent (where $\overline{X}$ denotes the complement of $X$).
Since $\phi(\idm{2}) = 2$, we must have
$\{\phi(\mathbbm{1}_X),\phi(\mathbbm{1}_{\overline{X}})\} = \{0,1\}$,
that is, $\mathcal{U}$ contains precisely one of $X$ and $\overline{X}$.

For $X \in \mathcal{U}$, define 
$f_X, g_X: I \rightarrow \{ 1, 2\}$ 
by
$$(f_X(i),g_X(i)) = \left \{
\begin{array}{l}
(2,1) \mbox{ if $i \in X$}, \\
(1,2) \mbox{ otherwise}.
\end{array}
\right.$$
Then $f_X$ is adjacent to $\idm{0}$ and $\mathbbm{1}_X$, thus
$\phi(f_X) = 2$. Since $g_X$  is adjacent to $\idm{0}$ and $f_X$,
we then have $\phi(g_X) = 1$. Now for any $Y \subseteq I$ containing $X$,
$\mathbbm{1}_{\overline{Y}}$ is adjacent to $g_X$, thus 
$\phi(\mathbbm{1}_{\overline{Y}}) = 0$ and $\phi(\mathbbm{1}_{Y}) = 1$. 
This shows that if $Y$ contains $X \in \mathcal{U}$, then $Y \in \mathcal{U}$.

For $X, Y \in \mathcal{U}$, define 
$f_{X\cap Y}, f_{X \setminus Y}, f_{\overline{X}}: I \rightarrow \{0, 1, 2\}$ by
$$(f_{X\cap Y}(i),f_{X \setminus Y}(i),f_{\overline{X}}(i)) = \left \{
\begin{array}{l}
(0,1,2) \mbox{ if $i \in X \cap Y$}, \\
(2,0,1) \mbox{ if $i \in X \setminus Y$}, \\
(1,2,0) \mbox{ if $i \in \overline{X}$}.
\end{array}
\right.$$
Then $f_{X\cap Y}, f_{X \setminus Y}$ and $f_{\overline{X}}$ are 
mutually adjacent. Therefore 
$$\{\phi(f_{X\cap Y}), \phi(f_{X \setminus Y}), \phi(f_{\overline{X}}) \} 
= \{0,1,2\}.$$ 
Since $f_{\overline{X}}$ is adjacent to $\mathbbm{1}_{\overline{X}}$
and $\phi(\mathbbm{1}_{\overline{X}}) = 0$, we have 
$\phi(f_{\overline{X}}) \neq 0$.
Similarly, $X \setminus Y \subseteq \overline{Y}$ whence 
$\phi(\mathbbm{1}_{X \setminus Y}) = 0$, and $\mathbbm{1}_{X \setminus Y}$ is
adjacent to $f_{X \setminus Y}$, so that $\phi(f_{X \setminus Y}) \neq 0$.
Therefore $\phi(f_{X\cap Y}) = 0$. Since $\mathbbm{1}_{X \cap Y}$ is adjacent to
$f_{X\cap Y}$, we then have $\phi(\mathbbm{1}_{X \cap Y}) = 1$, that is,
$X \cap Y \in \mathcal{U}$.
This shows that $\mathcal{U}$ is an ultrafilter.
\end{proof}
The above proof is essentially the correspondence between the $n$-colourings 
of powers of $\bK_n$, $n \geq 3$ and the 0-1 measures on a set
established by Greenwell and Lov\'asz~\cite{greenwelllovasz}.

\begin{cor}[L\"{a}uchli \cite{lauchli}] \label{uf<->knc}
For every $n \geq 3$, the ultrafilter axiom is equivalent 
to the statement that $\bK_n$ is compact.
\end{cor}

We can then expand the class of structures for which compactness is equivalent
to the ultrafilter axiom by using primitive positive definitions.
The same method also allows reductions amongst constraint satisfaction problems.
Let $\sigma$ be a type and $\bA$ a $\sigma$-structure with universe
$\{0, \ldots, n-1\}$, that is, the same universe as that of $\bK_n$. 
For a $\sigma$-structure $\bB$ with distinguished
elements $x, y$, the binary relation 
$R_{(\bB,x,y)} \subseteq \{0, \ldots, n-1\}^2$ is defined by
$$
R_{(\bB,x,y)} = \{ (\phi(x), \phi(y)) \  |\  \phi: \bB \rightarrow \bA 
\mbox{ is a homomorphism }\}.
$$
Then $\bK_n$ is said to be {\em primitively positively definable}
from $\bA$ if for some $(\bB,x,y)$,  the adjacency relation
$\neq$ of $\bK_n$ coincides with $R_{(\bB,x,y)}$.
For instance, if $\bB$ is an undirected path with three edges
and $x, y$ are its endpoints, then $R_{(\bB,x,y)}$ on the 
undirected cycle $\bA$ with five vertices is the adjacency 
relation of $\bK_5$.

\begin{prop} Let $\bA$ be a structure with universe
$\{0, \ldots, n-1\}$, where $n \geq 3$. If $\bK_n$ is primitively
positively definable from  $\bA$, then the ultrafilter axiom 
is equivalent to the statement that $\bA$ is compact.
\end{prop}
\begin{proof}
Suppose that $\phi: \bA^I_{\mathcal{F}} \rightarrow \bA$
is a homomorphism. Then $\phi$ is a map from the universe
of $\left(\bK_n\right)^I_{\mathcal{F}}$
to that of $\bK_n$. We will show that $\phi$ is a homomorphism
of $\left(\bK_n\right)^I_{\mathcal{F}}$ to $\bK_n$.

By hypothesis, the adjacency relation $\neq$ of $\bK_n$ coincides 
with some $R_{(\bB,x,y)}$. For each $i, j \in \{0, \ldots, n-1\}$
with $i \neq j$, we can fix $\psi_{(i,j)}: \bB \rightarrow \bA$
such that $\psi_{(i,j)}(x) = i$ and $\psi_{(i,j)}(y) = j$.

Now let $f, g$ be adjacent elements of the universe of 
$\left(\bK_n\right)^I_{\mathcal{F}}$, and put
\[
J_{(f,g)} = \{ i \in~I \ | \ f(i) \neq g(i)\}.
\]
Consider the map
$\psi_{(f,g)}: \bB \rightarrow \bA^I_{\mathcal{F}}$
defined by $\psi_{(f,g)}(z) = h$, where
$$
h(i) = \left \{ \begin{array}{l}
\psi_{(f(i),g(i))}(z) \mbox{ if $i \in J_{(f,g)}$,} \\
f(i) \mbox{ otherwise}.
\end{array} \right.
$$
Then $\psi_{(f,g)}$ is a homomorphism from $\bB$ to $\bA^I_{\mathcal{F}}$
since the set of coordinates on which all relations are preserved
contains $J_{(f,g)}$, which is a member of $\mathcal{F}$. Therefore
$\phi \circ \psi_{(f,g)}: \bB \rightarrow \bA$ is a homomorphism.
Since $R_{(\bB,x,y)}$ is the adjacency relation of $\bK_n$,
we then have 
$$
\phi(f) = \phi \circ \psi_{(f,g)}(x) \neq \phi \circ \psi_{(f,g)}(y) = \phi(g).
$$
This shows that $\phi$ is a homomorphism
of $\left(\bK_n\right)^I_{\mathcal{F}}$ to $\bK_n$.

Therefore, if $\bA$ is compact, then $\bK_n$ is compact, and the ultrafilter
axiom holds.
\end{proof}

Most structures are ``projective'' (see~\cite{nesetrilluczak}), hence they 
primitively positively define the complete graph on their universe.
Thus the method used to prove that most constraint satisfaction 
problems are NP-complete also shows that for most structures,
compactness implies the ultrafilter axiom. It would be interesting
to see whether the correspondence can be pushed further.
\begin{prob} \label{uf->npc}
Let $\mathbb{A}$ be a structure for which compactness
implies the ultrafilter axiom. Is the corresponding
constraint satisfaction problem necessarily NP-complete?
\end{prob}
The algebraic approach to the dichotomy conjecture
proposes a criterion for determining precisely which
constraint satisfaction problems are polynomial in terms
of polymorphisms of relational structures (see~\cite{siggers}).
To answer the above problem affirmatively, it would be sufficient
to show that when the criterion is not satisfied on a structure
$\bA$, then the compactness of $\bA$ implies the ultrafilter axiom.
However, the converse cannot be proved without proving
that NP is different from P, since some compactness
hypotheses are provably not equivalent to the ultrafilter axiom.  

\section{Axioms of set theory}
In this section we present a few axioms that are weaker 
than the ultrafilter axiom, but not provable from the axioms of Zermelo
and Fraenkel. The results of this section can be found in
``The Axiom of Choice'' by Thomas Jech~\cite{jech} and 
``Zermelo's  Axiom of Choice - Its Origins, Development and Influence''
by Gregory H. Moore~\cite{moore}. Our purpose in considering such axioms
is to obtain independence results without going into forcing theory.
We will use the following axioms.

\smallskip \noindent
{\sc Order extension principle.} Every partial ordering of a set $X$ can be 
extended to a linear ordering of $X$.

\smallskip \noindent
{\sc Ordering principle.} Every set can be linearly ordered.

\smallskip \noindent
{\sc Axiom of choice for finite sets.} For every set $X$ of nonempty finite sets,
there is a function $f: X \rightarrow \cup X$ such that
$f(x) \in x$ for all $x \in X$.

\smallskip
The ultrafilter axiom implies the order extension principle, which implies
the ordering principle, which implies the axiom of choice for finite sets. 
None of the implications is reversible.
For $n \geq 2$, the axiom of choice for $n$-sets is the following statement.

\smallskip \noindent
\choice{n}: For every set $X$ of sets of cardinality $n$, 
there is a function $f: X \rightarrow \cup X$ such that
for each $x \in X$, $f(x)$ is an element of $x$.

\smallskip The axiom of choice for finite sets implies the conjunction of
\choice{n} for all $n$, but the implication is not reversible. There are various
dependencies amongst the axioms of choice for finite sets. For instance, 
\choice{kn} implies \choice{n} for all $k$. Also, \choice{2} is equivalent to \choice{4},
but independent from \choice{3}. Research about such dependencies seems to
have been an active area up to the seventies, as indicated in the 
exercises to chapter 7 of \cite{jech}. It culminated in the following result.
\begin{prop}[Gauntt \cite{gauntt}]\label{comdep}
Let $m\geq 2$ be an integer and $S$ a set of integers. Then \choice{m} follows
from the conjunction of \choice{n}, $n \in S$ if and only if any fixed-point free
subgroup $G$ of the symmetric group $S_n$ contains a fixed-point free subgroup $H$
and a finite sequence $H_1, \ldots, H_k$ of proper subgroups of $H$
such that the sum of indices $[H:H_1] + \cdots + [H: H_k]$ belongs to $S$.
\end{prop}
The condition in Proposition~\ref{comdep} had been formulated by
Mostowski~\cite{mostowski}, who proved its sufficiency.
We note that \cite{gauntt} is only an announcement of the proof
of necessity. Apparently the proof has never been published,
but the result seems to be accepted by the community.
An earlier result states that \choice{m} follows from
the conjunction of \choice{k}, $2 \leq k \leq n$ if and only if
whenever $m$ is expressed as a sum of primes, one of the primes is at most $n$.
In particular, this implies that none of the statements \choice{k} can be proved 
from the axioms of Zermelo and Fraenkel.

One further axiom we need to consider is the following.

\smallskip \noindent
{\sc Kinna-Wagner Principle.} For every set $X$ of sets of cardinality at least
$2$, there is a function $f: X \rightarrow \mathcal{P}(\cup X)$ such that
for all $x \in X$, $f(x)$ is a nonempty proper subset of $x$.

\smallskip The Kinna-Wagner Principle implies the ordering principle,
but is independent from the ultrafilter axiom. We could not find
references to cardinality-specific versions of the Kinna-Wagner principle,
but these will be useful to us.

\smallskip \noindent
\kw{n}: For every set $X$ of sets of cardinality $n$, 
there is a function $f: X \rightarrow \mathcal{P}(\cup X)$ such that
for all $x \in X$, $f(x)$ is a nonempty proper subset of $x$.

Clearly, there is no loss of generality in assuming that $|f(x)| \leq |x|/2$
for all $x \in X$. Thus, \kw{2} is equivalent to \choice{2} and 
\kw{3} is equivalent to \choice{3}. In general, \kw{n} can be weaker than
\choice{n}, but since \choice{n} follows from \kw{n} and the conjunction
of all \choice{k}, $k \leq n/2$, \kw{n} is still independent from
the axioms of Zermelo and Fraenkel.

\section{Intermediate compactness results}
For $n \geq 2$, the directed $n$-cycle 
$\bC_n$ is the structure with universe $\mathbb{Z}_n = \{0,1,\ldots,n-1\}$
and one binary relation $R(\bC_n) = \{(i,i+1) \ | \ i \in \mathbb{Z}_n \}$. 
It has long been known that the constraint satisfaction problem associated 
with $\bC_n$ is polynomial for each $n$. In this section,
we show that the hypothesis that $\bC_n$ is compact is weaker
than the ultrafilter axiom, but not provable from the axioms
of Zermelo and Fraenkel.
\begin{prop}
For $n \geq 2$, \choice{n} implies that $\bC_n$ is compact.
\end{prop}
\begin{proof} Let ${\mathcal F}$ be a filter on a set $I$.
First note that for the equivalence relation
$\sim$ of Proposition~\ref{com=ftq}, 
$(\bC_n)^I_{\mathcal F}/\!\sim$
is a disjoint union of copies of $\bC_n$. 
Indeed, we have $f \rightarrow g$ if and only
if $\{ i \in I | g(i) = f(i) + 1 \} \in {\mathcal F}$
(where addition is modulo $n$). 

Let $\mathcal{C}$ be the set of connected components of
$(\bC_n)^I_{\mathcal F}/\!\sim$. For each $C \in \mathcal{C}$,
the set of homomorphisms from $C$ to $\bC_n$ has cardinality 
exactly $n$. If we assume \choice{n}, we can select a homomorphism
$\phi_C: C \rightarrow \bC_n$ for each $C \in \mathcal{C}$.
Then, $\phi = \bigcup_{C  \in \mathcal{C}} \phi_C$ is a homomorphism
of $(\bC_n)^I_{\mathcal F}/\!\sim$ to $\bC_n$. This implies that
$(\bC_n)^I_{\mathcal F}$ admits a homomorphism to $\bC_n$,
hence $\bC_n$ is compact by Proposition~\ref{com=ftq}.
\end{proof}

Thus, the hypothesis that $\bC_n$ is compact is weaker than the
ultrafilter axiom. The next two results show that it cannot be
proved from the axioms of Zermelo and Fraenkel.

\begin{prop}
Let $p$ be a prime integer. If $\bC_p$ is compact, then
\kw{p} holds. 
\end{prop}
\begin{proof} 
Let $X$ be a set of sets of cardinality $p$. Let
$I$ be the set of partial choice functions on $X$, that is,
$$
I = \{ c: Y \rightarrow \cup X \ | \ Y \subseteq X \mbox{ and }
c(x) \in x \mbox{ for all $x \in Y$}\}.
$$
For $x \in X$, put
$$
x^+ = \{ c \in I \ | \ \mbox{$x$  is in the domain of $c$}\}.
$$
Then the family $\{ x^+ \ | \ x \in X\}$ is closed under finite intersection,
hence it generates a filter $\mathcal{F}$ on $I$. By hypothesis,
there exists a homomorphism $\phi$ of 
$\left (\bC_p \right )^I_{\mathcal F}$ to $\bC_p$.
We will use $\phi$ to construct a function $f: X \rightarrow \mathcal{P}(\cup (X))$
such that for all $x \in X$, $f(x)$ is a nonempty proper subset of $x$.

For $x \in X$ and $j \in x$, put
$$
x^+_j = \{ c \in I \ | \ \mbox{$x$  is in the domain of $c$ and $c(x) = j$}\}.
$$
Then $I$ can be partitioned into the sets $x^+_j$, $j \in x$, and $x^-$,
where $x^-$ is the set of elements of $I$ that do not have $x$ in their domain.
Since $x$ is a set of cardinality $p$, there exist $p!$ bijections of
$x$ to $\{0, \ldots, p-1\}$. For a bijection $\psi: x \rightarrow \{0, \ldots, p-1\}$,
we let $f_{\psi}$ be the element of the universe of $(\bC_p)^I_{\mathcal F}$
defined by $f_{\psi}(i) = \psi(j)$ if $i \in x^+_j$, and $f_{\psi}(i) = 0$
if $i \in x^-$. Then $\{f_{\psi}, f_{\psi + 1}, \ldots, f_{\psi + p - 1} \}$
is a copy of $\bC_p$ in $(\bC_p)^I_{\mathcal F}$, where
$\psi + k: x \rightarrow \{0, \ldots, p-1\}$ is defined by $\psi + k(j) = \psi(j) + k$. 
Therefore $\phi$ maps 
$\{f_{\psi}, f_{\psi + 1}, \ldots, f_{\psi + p - 1} \}$ bijectively to the
universe of $\bC_p$. Thus there exists exactly one $k \in \{0, \ldots, p-1\}$
such that $\phi(f_{\psi + k}) = 0$. We call $(\psi +k)^{-1}(0) \in x$ 
the {\em element of $x$ distinguished by  
$\{\psi, \psi + 1, \ldots, \psi + p - 1 \}$}.

The $p!$ bijections of $x$ to $\{0, \ldots, p-1\}$ are partitionned
into $(p-1)!$ classes of the form $\{\psi, \psi + 1, \ldots, \psi + p - 1 \}$.
The number of times an element of $x$ appears as distinguished element
is not constant, since $p$ does not divide $(p-1)!$. Therefore the set
$y_x$ of elements of $x$ that appear the most times as distinguished
element is a nonempty proper subset of $x$. Therefore the function 
$f: X \rightarrow \mathcal{P}(\cup X)$ defined by $f(x) = y_x$
satisfies the required properties.
\end{proof}
The fact that \choice{2} = \kw{2} is equivalent to the compactness of
$\bC_2$ was proved by Mycielski~\cite{mycielski}. (Note that
$\bC_2 = \bK_2$.) The above results
show that \choice{3} = \kw{3} is equivalent to the compactness of
$\bC_3$. It is not hard to show that for prime $p$, \kw{p} is equivalent to
the compactness of $\bC_p$, and that for some composite numbers $n$,
\choice{n}, \kw{n} and the compactness of $\bC_n$ are inequivalent
hypotheses. However, for our purposes, it is sufficient to show
that the compactness of $\bC_n$ can never be proved from the axioms
of Zermelo and Fraenkel. Our next result completes the proof.

\begin{prop}
$\bC_n$ is compact if and only if $\bC_p$ is compact for every prime factor
$p$ of $n$. 
\end{prop}
\begin{proof} We first show that if $\bC_{kp}$ is compact, then
$\bC_{p}$ is compact.  Suppose that 
$\phi: \left (\bC_{kp} \right )^I_{\mathcal{F}} \rightarrow \bC_{kp}$
is a homomorphism. For $f$ in the universe of 
$\left (\bC_{p} \right )^I_{\mathcal F}$, we let $kf$ be the 
element of the universe of $\left (\bC_{kp} \right )^I_{\mathcal F}$
defined by $kf(i) = k\cdot f(i)$. We define 
$\psi: \left (\bC_{p} \right )^I_{\mathcal{F}} \rightarrow \bC_{p}$
by letting $\psi(f)$ be the unique $j \in \mathbb{Z}_p$
such that $\phi(kf) \in \{ kj, kj+1, \ldots, kj + k - 1\}$.
If $(f,g) \in R(\left(\bC_p\right )^I_{\mathcal F})$, then
there is a directed path of length $k$ from $kf$ to $kg$
in $\left (\bC_{kp} \right )^I_{\mathcal F}$. Therefore, 
$\psi(g) = \psi(f) + 1$. This shows that $\psi$ is a homomorphism.
Therefore if $\bC_{kp}$ is compact, then $\bC_{p}$ is compact.

We next show that if $p$ and $q$ are relatively prime,
and $\bC_p$, $\bC_q$ are compact, then $\bC_{pq}$ is compact. 
This easily follows from the fact that $\bC_{pq}$ is naturally 
isomorphic to the product $\bC_p \times \bC_q$ with universe
$\mathbb{Z}_p \times \mathbb{Z}_q$ and relation 
$R(\bC_p \times \bC_q) =
\{((i,j),(i+1,j+1)) \ | \ (i,j) \in \mathbb{Z}_p \times \mathbb{Z}_q \}$.
Thus, a homomorphism $\phi$ to $\bC_{pq}$ naturally corresponds
to a pair $\phi_p$, $\phi_q$ of homomorphisms to $\bC_p$
and $\bC_q$ respectively. The definition of compactness then
directly implies that if $\bC_p$ and $\bC_q$ are compact, then 
$\bC_{pq}$ is compact.

It only remains to show that if $p$ is prime and $\bC_{p}$ is compact,
then $\bC_{p^k}$ is compact for all $k \geq 1$. We proceed
by induction on $k$. Suppose that $\bC_{p}$ and $\bC_{p^k}$
are compact. There is a homomorphism $\phi$ of $\bC_{p^{k+1}}$ to
$\bC_{p^k}$ corresponding to the natural quotient of
$\mathbb{Z}_{p^{k+1}}$ to $\mathbb{Z}_{p^{k}}$.
Applying $\phi$ coordinatewise yields a homomorphism
$\phi': \left (\bC_{p^{k+1}} \right )^I_{\mathcal F}
\rightarrow \left (\bC_{p^{k}} \right )^I_{\mathcal F}$. If $\bC_{p^k}$
is compact, there exists a homomorphism
$\phi'': \left (\bC_{p^{k}} \right )^I_{\mathcal F} \rightarrow \bC_{p^k}$.
Thus $\phi'' \circ \phi': \left (\bC_{p^{k+1}} \right )^I_{\mathcal F}
\rightarrow \bC_{p^k}$ is a homomorphism. The natural quotient by the
equivalence $\sim$ of Proposition~\ref{com=ftq} yields
a homomorphism $\psi:  \left (\bC_{p^{k+1}} \right )^I_{\mathcal F}/\!\sim
\; \rightarrow \bC_{p^k}$. The connected components of
$\left (\bC_{p^{k+1}} \right )^I_{\mathcal F}/\!\sim$
are copies of $\bC_{p^{k+1}}$. Each of these connected
components admits exactly $p$ homomorphisms to $\bC_{p^{k+1}}$
with the property that their composition with $\phi$ coincides
with $\psi$. We need to choose one of these homomorphisms
for each connected component.

We cannot prove that the compactness of $\bC_{p}$ implies
\choice{p} for $p$ larger than $3$, but in the present
context, the full force of \choice{p} is not required.
Let $\bA$ be the structure with universe $\psi^{-1}(0)$,
and one binary relation consisting of the pairs $(f,g)$
such that there is a directed path of length $p^k$ from
$f$ to $g$ in $\left (\bC_{p^{k+1}} \right )^I_{\mathcal F}/\!\sim$.
The connected components of $\bA$ are copies of $\bC_{p}$,
one in each connected component of 
$\left (\bC_{p^{k+1}} \right )^I_{\mathcal F}/\!\sim$.
Similarly, let $\bB$ be the structure with universe $\phi^{-1}(0)$,
and one binary relation consisting of the pairs $(i,j)$
such that there is a directed path of length $p^k$ from
$i$ to $j$ in $\bC_{p^{k+1}}$. Then $\bB$ is a copy of $\bC_{p}$.
Since $\bC_{p}$ is compact, there exists a homomorphism
$\chi^*: \bA \rightarrow \bB$. We can then define
$\chi:  \left (\bC_{p^{k+1}} \right )^I_{\mathcal F}/\!\sim \;
\rightarrow \bC_{p^{k+1}}$
componentwise as the unique homomorphism with the property
that $\chi(f) = \chi^*(f)$ for all $f \in \psi^{-1}(0)$.
\end{proof}

The results above imply that the conjunction of \choice{n} for all $n$
is equivalent to the compactness of $\bC_n$ for all $n$. There are
other structures that can be proved to be compact from
the conjunction of \choice{n} for all $n$. However it seems 
unlikely that this is the case for all structures that are 
not compactness complete. We conclude this section with one
possible witness to this assertion. 
\begin{prop} \label{oep->c}
Let $\bA$ be the structure on the universe $\{0, 1\}$ with
the two binary relations $\neq$ and $\leq$. The order extension
principle implies that $\bA$ is compact.
\end{prop}
\begin{proof}
Let $\mathcal{F}$ be a filter on a set $I$. We will show that
the order extension principle implies that $\bA^I_{\mathcal F}/\!\sim$
admits a homomorphism to $\bA$, where $\sim$ is the equivalence
of Proposition~\ref{com=ftq}. 
The elements of $\bA^I_{\mathcal F}$
will be represented by their characteristic function. In 
$\bA^I_{\mathcal F}/\!\sim$, the only element adjacent (under $\neq$) to
$\mathbbm{1}_X/\!\sim$ is $\mathbbm{1}_{\overline{X}}/\!\sim$.
The relation $\leq$ of $\bA^I_{\mathcal F}$ is characterised
by 
$${ \mathbbm{1}_X/\!\sim \; \leq \mathbbm{1}_Y/\!\sim } \Leftrightarrow
{\{ i \in I \ | \ \mathbbm{1}_X(i) \leq \mathbbm{1}_Y(i) \} \in \mathcal{F}.}$$
The order extension principle implies that it extends to a linear order
$\preceq$ on $\bA^I_{\mathcal F}/\!\sim$. We define 
$\phi: \bA^I_{\mathcal F}/\!\sim \; \rightarrow \bA$ by
$$
\phi(\mathbbm{1}_X/\!\sim) = \left \{
\begin{array}{l}
\mbox{$0$ if $\mathbbm{1}_X/\!\sim \; \preceq \mathbbm{1}_{\overline{X}}/\!\sim$,} \\
\mbox{$1$ otherwise.}
\end{array}
\right.
$$
For all $\mathbbm{1}_X/\!\sim$, 
we have $\{\phi(\mathbbm{1}_X/\!\sim), \phi(\mathbbm{1}_{\overline{X}}/\!\sim)\}
= \{ 0, 1\}.$ Therefore $\phi$ preserves the adjacency relation $\neq$.
For $\mathbbm{1}_X/\!\sim \; \leq \mathbbm{1}_Y/\!\sim$, if $\phi(\mathbbm{1}_Y/\!\sim) = 0$,
then $\mathbbm{1}_X/\!\sim \; \leq \mathbbm{1}_Y/\!\sim \; \preceq \mathbbm{1}_{\overline{Y}}
\; \leq \mathbbm{1}_{\overline{X}}$, hence $\phi(\mathbbm{1}_X/\!\sim) = 0$.
Therefore $\phi$ preserves $\leq$. This shows that $\phi$ is a homomorphism. 
\end{proof}
We have not been able to prove or disprove that the compactness
of the structure $\bA$ in Proposition~\ref{oep->c} follows from the
conjunction of \choice{n} for all $n$. On the other hand,
the hypothesis that $\bA$ is compact seems to be fairly close
to the order extension principle: It implies that
every order extends to a relation $R$ that is trichotomic
and has the property that if $(x,y)$, $(y,z)$, $(z,x)$ belong
to $R$, then $\{x,y,z\}$ is an antichain in the order.
Again we cannot prove or disprove that the latter property is
weaker than the order extension principle. Perhaps the trick
of comparing to previously investigated axioms has its limits,
and forcing theory would be needed at some point.

\section{Conclusion} We have established a partial correspondence 
between the hierarchy of strength of compactness hypotheses 
and the complexity hierarchy of constraint satisfaction problems. A positive
answer to Problem~\ref{uf->npc} would strenghten this correspondence.
The conjectured algebraic criterion for characterising NP-complete
constraint satisfaction problems may be the key to settling this question.
Polynomial constraint satisfaction problems do not all correspond to
equivalent compactness hypotheses. This is a bit unsettling. Is there
an algorithmic significance to the varying strength of compactness hypotheses?
More significantly, is there a possible interplay between descriptive
complexity and axiomatic set theory that goes beyond a simple correspondence?

\medskip {\bf Acknowledgements.} The authors wish to thank
Jaroslav Ne\v{s}et\v{r}il for pointing out the result
in reference \cite{greenwelllovasz}.  The authors gratefully acknowledge the partial support of NSERC for this work.

\end{document}